\documentclass{article} 

\usepackage{url}
\usepackage{subcaption}
\usepackage{amsmath,amssymb,graphicx,algorithm,algorithmic,bm,amsthm}
\usepackage{epstopdf}
\usepackage{color}
\usepackage{enumitem}
\usepackage{mathtools}

\newcommand{\R}{\mathbb{R}}
\newcommand{\Q}{\mathbb{P}}

\newcommand{\p}{\prime}

\newtheorem{theorem}{Theorem}[section]
\newtheorem{lemma}[theorem]{Lemma}

\title{Online Ranking with Constraints: A Primal-Dual Algorithm and Applications to Web Traffic-Shaping}

\author{Parikshit~Shah \qquad Akshay~Soni \qquad Troy~Chevalier} 

\begin{document} 
\maketitle

\begin{abstract}

We study the online constrained ranking problem motivated by an application to web-traffic shaping: an online stream of sessions arrive in which, within each session, we are asked to rank items. The challenge involves optimizing the ranking in each session so that local vs. global objectives are controlled: within each session one wishes to maximize a reward (local) while satisfying certain constraints over the entire set of sessions (global). A typical application of this setup is that of page optimization in a web portal. We wish to rank items so that not only is user engagement maximized in each session, but also other business constraints (such as the number of views/clicks delivered to various publishing partners) are satisfied. 

We describe an online algorithm for performing this optimization. A novel element of our approach is the use of linear programming duality and connections to the celebrated Hungarian algorithm. This framework enables us to determine a set of \emph{shadow prices} for each traffic-shaping constraint that can then be used directly in the final ranking function to assign near-optimal rankings. The (dual) linear program can be solved off-line periodically to determine the prices. At serving time these prices are used as weights to compute weighted rank-scores for the items, and the simplicity of the approach facilitates scalability to web applications. We provide rigorous theoretical guarantees for the performance of our online algorithm and validate our approach using numerical experiments on real web-traffic data from a prominent internet portal.
\end{abstract}





\section{Introduction}
This paper investigates the online constrained ranking problem --- a collection of sessions arrive in an online manner. In each session we wish to make optimal decisions that balance local versus global trade-offs: locally we wish to maximize the engagement of that session, while globally we wish to satisfy certain constraints over the entire set of sessions. Each session involves a user interacting with a collection of items, and the decision-making task at hand is that of selecting a ranking of these items. The rankings chosen across all sessions influence the rewards collected in the sessions as well as whether or not the overall constraints are satisfied. 
%
%

While our main focus in this paper is to develop algorithms for the general online constrained ranking problem, the key motivating application involves webpage layout optimization and traffic-shaping. In this context, each session corresponds to a user interacting with a webpage (for example a web portal), and the decision-maker wishes to optimize the placement of content from an assortment of content (e.g. news articles, sports articles, ads, blogs, financial news, etc.). 

The webpage is assumed to be divided into slots, and the problem of assigning content to slots may be viewed as an assignment or ranking problem. The natural locally optimal approach to optimizing a webpage involves computing an engagement score (e.g., a click probability or a predicted dwell-time) for each content, and displaying the content within slots on the page in descending order of the scores  so that the highest (predicted) engaging content is in the most highly clicked slot, and so on. 

However, in practice there are often other business constraints in place. For example we may wish to deliver a certain guaranteed number of clicks to different segments of the traffic. Business considerations may dictate that a certain number of clicks be delivered to a premium advertiser, or that a certain number of top-ranked page-views be delivered to certain premium news-agency partners over the entire collection of sessions in the traffic. These additional side constraints introduce trade-offs; performing a simple "locally optimal" ranking may violate these constraints. The challenge thus is to pick assignments in each session, online, in a way that satisfy the traffic-wide business constraints while maximizing user engagement. Since the traffic distribution is quite well-known in advance (and data is available for the same), the natural question to ask is whether we can optimize the pages in each session to maximize user engagement while satisfying the other traffic constraints. Our paper deals with developing a principled approach to deal with this traffic-shaping problem.

Note that our paper does not seek to address the learning-to-rank problem \cite{LTR,LTR_online}, i.e. the problem of learning how to predict user engagement from data. Indeed, in many practical systems, a natural separation of concerns is assumed between learning-to-rank (or machine learning for predicting other user-engagement signals such as click-through rate, dwell-time etc.) and the traffic-shaping problem. The machine learned models are used as input signals to a \emph{federation} layer that performs traffic-shaping, i.e. ingesting the signals and optimizing the ranking to maximize engagement and satisfying the constraints. However, the algorithm for traffic-shaping itself is completely agnostic to the details of how these signals are learned.

\textbf{Comparison to related work:} The traffic-shaping and click-shaping problems were studied by Agarwal et al. \cite{Agarwal_cs1,Agarwal_cs2,Agarwal_cs3} using an optimization based framework that involved probabilistic sampling. Several important features distinguish our work. The aforementioned work deals only with the single slot case (i.e. each session involves choosing a single item from a collection and picking the best one that satisfies a constraint), and they defer the multi-slot problem as an open problem. While they also adopt a linear programming approach, their model is substantially different; their optimization variables are probabilities of sampling items that maximize reward and satisfy traffic constraints. Moreover, they do not present theoretical guarantees on the solution quality. In contrast, our approach directly addresses the multi-slot situation (i.e. ranking items), our optimization involves optimizing over the set of permutations (i.e. assignments), and we tackle the online aspect of the problem by incorporating a learning phase followed by an online decision-making phase. We evaluate our method empirically and provide rigorous theoretical guarantees. 

Our work draws on ideas and tools from the online optimization literature, specifically the primal-dual method for online linear programming \cite{Ye} and its applications to online combinatorial optimization problems \cite{Buchbinder} such as the ad-words problem \cite{Mehta,Devanur}, whole-page optimization problem for ads \cite{Devanur2}, combinatorial auctions \cite{comb_auc}, and extensions to the nonlinear case \cite{Chen}. While our approach substantially uses ideas and tools from this literature, a key distinction is the nature of the decision-making problem. Each online session in the aforementioned lines of work involves a binary decision-making problem, i.e. whether or not to assign a unit of an item in the session. In contrast, in our problem, each session involves choosing a permutation (more specifically a perfect matching between the items and the slots). 

Moreover, unlike \cite{Ye} where the decision-maker must satisfy budget (upper bound) constraints, in our paper we have allocation (lower bound) constraints wherein a certain number of units (for e.g. clicks) must be delivered. Finally, we mention that within the literature, different lines of work arise from different distributional assumptions made on the sessions; e.g. \cite{Mehta} analyze the problem in the setting where the ordering of sessions may be adversarial with respect to the online decision maker. Another line of work \cite{Devanur, Ye} assumes the \emph{random permutation model}, wherein the distribution order of sessions may be assumed to exchangeable. Different theoretical guarantees are possible under different scenarios, and our work employs the random permutation model also.
Our main contributions in this paper are the following:
\begin{itemize}[leftmargin=*]
\item \textbf{Solution Form:} We present a new algorithm for the online constrained ranking problem. In the initial learning phase, dual prices $\lambda_t$ corresponding to each traffic constraint $t$ are learned. In the subsequent online phase, in each session, one is provided matrices $C$ and $A_{t}$ (see Sec. \ref{sec:formulation} for details) that capture the predictions for engagement and traffic-shaping units contributed for different possible rankings. The online ranking algorithm assumes the simple form:
\begin{equation} \label{eq:form}
\sigma = \texttt{MaxWeightMatching}\left(C + \sum_{t=1}^T \lambda_t A_t \right).
\end{equation}
When $\lambda_t=0$, one recovers the ``greedy'' solution, i.e. that of optimizing each session individually and disregarding the constraints. In order to satisfy the constraints, the session rankings must be re-optimized (over the space of permutations) as per \eqref{eq:form}.
As a consequence of the connection to maximum-weight matching, we intimately use the Hungarian algorithm \cite{Hungarian} and its analysis. 
\item \textbf{Scalability:} In the learning phase, data is collected and a single  linear program needs to be solved offline to obtain the prices. Thereafter, the online aspect is extremely scalable as it involves computing a matching using the prices. 
\item \textbf{Evaluation and Guarantees:} We present detailed empirical validation of our algorithm on real data. We also theoretically analyze our algorithm using the primal-dual method. Specifically we show that our algorithm achieves a competitive ratio of $1-O(\epsilon)$, where $\epsilon$ is the fraction of samples that are used in the learning phase.
\item \textbf{Flexibility:} Another desirable feature of our algorithm is that it is flexible; indeed constraints can change, performance of the machine learned input signals can vary, and we would like a solution that is robust. Robustness is achieved by simply repeating the learning phase (with the new constraints in place, for instance), and obtaining the new prices. Thereafter, the implementation of the online algorithm remains the same, only the new prices need be propagated.
\end{itemize}

\textbf{Notation:} Throughout the paper, we will use the notation $[m]:=\left\{1, \ldots, m \right\}$; this quantity will refer to the size of the list (of documents) to be ranked in each session. A permutation (or ranking), denoted by $\sigma$ refers to a bijection between sets $\sigma: [m] \rightarrow [m]$. Usually the sets of interest are the set of documents (in a fixed session), and the different slots, and a permutation $\sigma$ specifies the assignment of documents to slots. We will use the terms permutation, assignment, and ranking interchangeably throughout the paper. It will often be convenient to represent the ranking $\sigma$ by a permutation matrix $P \in \R^{m \times m}$ where $(P)_{ij}=1$ if $\sigma(i)=j$ and $0$ otherwise. Such a matrix has a single entry equal to one in each row and column, with the remaining entries equal to zero. We will refer to the set of $m!$ permutation matrices as $\mathcal{P}$. The convex hull of $\mathcal{P}$, denoted by $co(\mathcal{P})$ is the Birkhoff polytope \cite{Hungarian}, which will play an important role. 

Throughout the paper, the index $k$ is reserved for sessions, i.e. $k \in [n]$, the index $t$ is reserved for the $T$ different traffic-shaping constraints, i.e. $t \in [T]$. In a session $k$, if document $d$ is shown in position $p$, a user engagement of $\left(C_k \right)_{d,p}$ is assumed to be accrued. Given a permutation $\sigma$, the overall engagement reward is assumed to be additive, i.e. given by $\sum_{i=1}^m \left(C_k \right)_{i,\sigma(i)}$. The matrix $C_k$ is conveniently represented by a bipartite graph with $m$ nodes on the left and right and the edges having weights $\left(C_k \right)_{d,p}$. A matching $\sigma$ achieves a weight of $\sum_{i=1}^m \left(C_k \right)_{i,\sigma(i)}$. A maximum-weight perfect matching is a permutation $\sigma \in \mathcal{P}$ with maximum possible weight. We remind the reader that a maximum-weight perfect matching in a bipartite graph is obtained via the Hungarian algorithm \cite{Hungarian}. We will use the notation $P \geq 0$ to specify that the matrix $P$ is entry-wise non-negative, and $P^\p$ denotes the matrix transpose. We will use $\langle P,Q \rangle := \sum_{i=1}^{m} \sum_{j=1}^m P_{ij}Q_{ij}$ to denote the inner-product between the matrices $P, Q \in \R^{m \times m}$. When a ranking $P$ is chosen, the engagement corresponding to that ranking is given by $\langle C_k, P \rangle$ due to the additive nature of the reward.



\section{Formulation} \label{sec:formulation}
Let $k \in [n]$ index a collection of sessions where different users interact with the content-delivery system of interest. In each session, a collection of $m$ documents are made available. The task of the ranking system is, in each session $k$, to rank the documents, i.e. produce a permutation $\sigma_k: [m] \rightarrow [m]$.  

For each session, we assume access to predictions of how each document would perform with respect to user engagement (e.g. clicks or dwell time) if document $d$ was shown in the $p^{th}$ position. Let $C \in \R^{m \times m}$ be the matrix such that $C_{d,p}$ models the engagement when document $d$ is shown in the $p^{th}$ position. (Typically, this type of information is available from the output of a machine-learned model that predicts the engagement for each document, personalized for each user corresponding to the session in question). For each session, the engagement is given by $\sum_{p=1}^{m} C_{p,\sigma(p)} = \langle C, P \rangle$, where $P$ is the permutation matrix corresponding to $\sigma$. The total engagement across all sessions is therefore $\sum_{k=1}^{n} \langle C_k, P_k \rangle$, where $k$ is the user engagement for the $k^{th}$ session and $P_k$ is the corresponding permutation chosen.

Simultaneously, we assume that a number of traffic-shaping constraints are present. As an example, we may have constraints on the number of clicks that must be delivered to various publishing partners (in a fixed number of sessions). These can be captured as:
$$
\sum_{k=1}^{n} \langle A_{kt}, P_k \rangle \geq b_t,
$$
where the index $k$ refers to the session, and $t$ is refers to $t^{th}$ traffic-shaping constraint. The matrix  $A_{kt} \in \R^{m \times m}$ is the matrix whose $(d,p)$ entry captures the number of engagement units (e.g. clicks) delivered for the $t^{th}$ constraint in the $k^{th}$ session when document $d$ is shown in position $p$ (such estimates are also typically obtained as the output of a machine-learned click model)\footnote{As an aside, we note that in practice, $C_k$ and $A_{kt}$ are predictions obtained from a machine learned model, and the engagement (i.e., dwell-time and clicks) realized will likely be different from the predictions. We ignore this distinction in this paper, and make the simplifying assumption that the predictions are ``perfect''. When the models are consistent and a sufficiently large number of sessions are involved, the expected performance (being optimized here) will be close to that of the realized one.} 
. The constant $b_t$ is the number of clicks that are committed to the publishing partner in $n$ sessions contractually.
The objective of the ranking system is to maximize the overall user-engagement while satisfying certain constraints.

The above description suggests a natural optimization formulation:

\begin{equation}
\begin{split} \label{opt:P0}
\underset{P_1, \ldots, P_n}{\text{maximize}} & \qquad \sum_{k=1}^{n} \langle C_k, P_k \rangle \\
\text{subject to}  & \qquad \sum_{k=1}^{n} \langle A_{kt}, P_k \rangle \geq b_t \; \; t=1, \ldots, T\\
& \qquad P_k \in \mathcal{P} \; \; k=1, \ldots, n.
\end{split}
\end{equation}

While conceptually clean, working directly with this formulation has several disadvantages:
\begin{itemize}[leftmargin=*]
\item Solving this formulation requires knowledge of the $C_k, A_{kt}$ before-hand, whereas the problem is online in nature (i.e. the sessions arrive in an online manner).
\item The above optimization formulation is combinatorial and therefore intractable, i.e. it involves optimizing over the set of permutations. 
\end{itemize}

To address the last point, we perform a \emph{convex relaxation} \cite[Chapter 6]{Bubeck} of the problem, i.e. relax the optimization constraints $P_k \in \mathcal{P}$ to $P_k \in co(\mathcal{P})$, where the $co(\mathcal{P})$ represents the convex hull of $\mathcal{P}$ \cite[Chapter 18]{Hungarian}. The resulting relaxation converts the optimization formulation \eqref{opt:P0} into a convex optimization problem (indeed, it is a linear programming problem), which can therefore be solved (in principle) in polynomial time.
There are a number of difficulties pertaining to the relaxation:
\begin{itemize}[leftmargin=*]
\item Linear programming is computationally intensive and solving one at serving-time is often infeasible due to the strict latency requirements. Hence we desire to avoid solving them at run-time (although solving offline on a periodic basis is acceptable).
\item A problem associated with convex relaxations of combinatorial problems is that the resulting solution may be fractional (i.e. the solutions $P_k$ may not be permutation matrices). We then have to devise a scheme to convert the optimal solution to a permutation matrix via a rounding scheme. As we will show in our subsequent analysis, we will not need to round solutions. The special structure of our problem and well-known results from graph matching and the analysis of the Hungarian algorithm \cite[Chapter 17]{Hungarian} actually guarantee that we can always produce a valid extremal solution, rather than a fractional one.
\item In order to solve the convex relaxation, one needs a tractable representation of the set $co(\mathcal{P})$ in terms of a small number of equations and inequalities. It turns out that the convex hull of the set of permutation matrices has a compact description via the Birkhoff-von Neumann Theorem \cite[Chapter 18]{Hungarian}, it is simply the set of so-called doubly stochastic matrices (sometimes also called the Birkhoff polytope) given by:
\begin{equation}
\mathcal{B}:=co(\mathcal{P})=\left\{ P \in \R^{m \times m} \; | \; P \mathbf{1} = \mathbf{1} \; P^{\prime} \mathbf{1} = \mathbf{1} \; P \geq 0 \right\}.
\end{equation}
\end{itemize}

We will be able to resolve the above difficulties by computing and analyzing the dual; indeed the analysis will reveal a natural and simple online algorithm which we will describe below.

We first explicitly state the primal and dual linear programs that result from the convex relaxation of \eqref{opt:P0}.

\begin{equation}
\begin{split} \label{opt:P1}
\underset{P_1, \ldots, P_n}{\text{maximize}} & \qquad \sum_{k=1}^{n} \langle C_k, P_k \rangle \\
\text{subject to}  & \qquad \sum_{k=1}^{n} \langle A_{kt}, P_k \rangle \geq b_t \; \; \forall \; t=1, \ldots, T\\
& \qquad P_k \in \mathcal{B} \; \; k=1, \ldots, n
\end{split}
\end{equation}

\begin{equation}
\begin{split} \label{opt:D1}
\underset{\lambda, \alpha_k,\beta_k}{\text{maximize}} & \qquad \sum_{t=1}^T \lambda_t b_t + \sum_{k=1}^n \alpha_k^{\prime} \mathbf{1} + \sum_{k=1}^n \beta_k^{\prime} \mathbf{1} \\
\text{subject to}  & \qquad C_k + \sum_{t=1}^{T} \lambda_t A_{kt} + \mathbf{1} \alpha_k^{\prime}+  \beta_k \mathbf{1}^{\prime} \leq 0 \; \; \forall k=1, \ldots, n \\
& \qquad \lambda \geq 0
\end{split}
\end{equation}

The dual linear program is of key interest in the paper --- our algorithm will consist of solving \eqref{opt:D1} on historical data, and computing the optimal dual variables $\lambda$ which have a natural price interpretation. We describe this in the next section.

\section{The Algorithm} \label{sec:algorithm}
Recall that in our setup, the $n$ sessions arrive in a sequential manner online, and for each session $k$ we have access to the matrices $C_k$ (the engagement), and $A_{tk}$ (corresponding to the traffic-shaping constraint). We denote the full set of sessions by $N$.
Our approach involves using a subset of the sessions, $S$, initially to learn the dual prices. The corresponding primal/dual linear programs are called the \emph{sample linear programs}. In these learning-phase sessions, we make arbitrary decisions --- concretely we let $P_k = I$, the identity permutation.  Let $\hat{n}:=|S|$ be chosen and denote $\epsilon = \frac{\hat{n}}{n}$. Since the constraints require that a budget of $b_t$ clicks are to be delivered over $n$ sessions, it follows that when $\epsilon n$ sessions are present, the corresponding commitment is only $\epsilon b_t$ clicks for each traffic-constraint $b_t$. In this way, the $b_t$ is rescaled. In order to be conservative, we instead require that for each constraint:
$
\sum_{k=1}^{\hat{n}} \langle A_{kt}, P_k \rangle \geq \nu \epsilon b_t.
$
The additional $\nu$ multiplicative factor ensures that the solution obtained is robust to random variations (we will discuss this in Section \ref{sec:theory}), and also to counter the conservative estimate that since the first $\epsilon$ fraction of sessions are used for learning, zero (or negligible) contributions to the constraints are made from it. Accordingly, the primal/dual sample linear programs are:
\begin{equation}
\begin{split} \label{opt:P2}
\underset{P_1, \ldots, P_{\hat{n}}}{\text{maximize}} & \qquad \sum_{k=1}^{\hat{n}} \langle C_k, P_k \rangle \\
\text{subject to}  & \qquad \sum_{k=1}^{\hat{n}} \langle A_{kt}, P_k \rangle \geq \nu \epsilon b_t \; \; \forall \; t=1, \ldots, T\\
& \qquad P_k \in \mathcal{B} \; \; k=1, \ldots, \hat{n}
\end{split}
\end{equation}
\begin{equation}
\begin{split} \label{opt:D2}
\underset{\lambda, \alpha_k, \beta_k}{\text{maximize}} & \qquad \sum_{t=1}^T \nu\epsilon \lambda_t b_t + \sum_{k=1}^{\hat{n}} \alpha_k^{\prime} \mathbf{1} + \sum_{k=1}^{\hat{n}} \beta_k^{\prime} \mathbf{1} \\
\text{subject to}  & \qquad C_k + \sum_{t=1}^{T} \lambda_t A_{kt} + \mathbf{1} \alpha_k^{\prime}+  \beta_k \mathbf{1}^{\prime} \leq 0 \; \; \forall k=1, \ldots, \hat{n} \\
& \qquad \lambda \geq 0
\end{split}
\end{equation}
In the theoretical anaylsis section we prove that when $\nu = 1+ 4 \epsilon$, with high probability a feasible solution is returned. We also prove that when $\nu = 1 - \epsilon$, an almost feasible solution is returned, and that the objective is almost as large as the hindsight optimal feasible solution.

\begin{algorithm}[!ht]
   \caption{Algorithm for optimal traffic-shaping}
   \label{alg:traffic_shaping}
\begin{algorithmic}[1]
   \STATE {\bfseries Input:} Data $\left\{ (C_k, A_{k1}, \ldots, A_{kT})\right\}$ for sessions $k=1, \ldots, n$ presented online for ranking, sample linear program parameters $\nu$, $\epsilon$.
   \STATE For the first $\hat{n}$ sessions, let $P_k = I$ (arbitrary decisions) and log the session data $\left\{ (C_k, A_{k1}, \ldots, A_{kT})\right\}$ corresponding to this sample set $k \in S$.
   \STATE Solve the sample linear program \eqref{opt:D2} on $S$ to obtain the dual prices $\hat{\lambda}_1, \ldots, \hat{\lambda}_m$.
   \WHILE{$\hat{n} < k \leq n$}
   \STATE Compute document score matrix $S_k = C_k + \sum_{t=1}^T \hat{\lambda}_t A_{kt}$ where $S_{d,p}$ represents the score when document $d$ is shown in position $p$. \label{step:scoring}
\STATE Set $\sigma_k:=\texttt{MaxWeightMatching}(S_k)$. \label{step:hungarian}
   \STATE {\bfseries Output:} Ranking $\sigma_k$ for session $k$. 
   \STATE Increment session counter $k:=k+1$
   \ENDWHILE
\end{algorithmic}
\end{algorithm}

In words, the algorithm works as follows. First, using historical data, we offline solve the linear program \eqref{opt:D2} to obtain the optimal dual prices $\hat{\lambda}_t$ for each traffic shaping constraint. These prices are then made available to the online serving system for ranking. Once the learning phase is over, upon the arrival of new sessions, the matrices $C_k, A_{kt}$ are computed, and the score matrix $S_k$ is computed as per step \ref{step:scoring}. As mentioned before, the $(d,p)$ entry of the matrix $S_k$ represents the score when document $d$ is shown in position $p$.

Once the score matrix is computed, we call a subroutine to compute the maximum-weight bipartite perfect matching with respect the weight matrix $S_k$, this perfect matching assigns documents to positions, and thereby yields the desired ranking $\sigma_k$ for that session. 


The run-time complexity is essentially the complexity of computing a perfect matching, for which a number of options are available. In order to compute a max-weight perfect matching exactly, one may use the celebrated Hungarian algorithm. The computational complexity of implementing the same is $O(m^3 \log m)$ \cite{Hungarian}. In scenarios of interest to us, we are typically only interested in ranking of the order of $20-200$ webpages in each session, and hence the algorithm is feasible to implement at run-time.  When the scale is larger and the Hungarian algorithm is not a viable option, one can instead implement the naive greedy algorithm\footnote{The greedy algorithm is simply the following: sort the edges of the bipartite graph in decreasing order. We maintain a set of edges $M$ (initialized to be empty), and loop over the sorted edges. In each step, add the corresponding edge to $M$ if it is possible to do so while maintaining the requirement that $M$ is a (possibly partial) matching.} as an alternative. The naive greedy algorithm is known to be a $\frac{1}{2}$-optimal algorithm for maximum matching, with a worst-case running time of $O(m^2)$. In our experiments, we use a position effects model called the ``reference CTR'' model, in which the greedy algorithm is actually optimal. The offline computational complexity is polynomial in the size of the problem --- it involves the solution of the linear program \eqref{opt:D2}, which scales with the sample size of the sampled linear program and is hence high. However, since offline computation is not a bottleneck, our approach is scalable to web applications.

While Algorithm \ref{alg:traffic_shaping} advocates solving the sampled linear program once (in advance) on historical data, this is mostly to simplify the analysis of the algorithm. In practice, we recommend solving the dual linear program periodically (e.g. once daily) to refresh the prices to capture the random as well as seasonal variations in traffic. The analysis of periodically recomputing the solution is also possible (in the spirit of \cite{Ye}), but the analysis is more complicated with only marginally better bounds --- we leave this for future work.

\section{Numerical Experiments} \label{sec:expts}
\subsection{Experimental Setup}
In this section we validate our algorithm via numerical experiments. Our numerical experiments are conducted on real data collected from session history on a major web-portal. We describe the setup for our experiment below.

Our data consists of $2000$ sessions of web-traffic data. In each session, we are presented with $20$ documents to rank on the web-page. Each document has associated with it the following attributes:
\begin{enumerate}[leftmargin=*]
\item The predicted dwell-time: This is a score generated by a machine-learned model on a large corpus of data that incorporates document features, user features, and document popularity signals. The predicted score is to be interpreted as a surrogate for the expected dwell-time when the user is shown that document. 
\item The click-through rate: This is again generated via a machine learned model, and is an estimate for the click-probability for a particular user to click on the document in question when that document is shown in the first position.
\item Newsiness score: This is a normalized score that captures how newsy a particular document is. It is also obtained as the output of a machine-learned model. When a newsy article is shown in a highly clickable slot on the page,  the overall page is deemed more newsy.
\item Publisher A: This is a binary $\left\{ 0,1 \right\}$ score that captures whether a document is curated by Publisher A or not. Showing a Publisher A-curated document in a more clickable slot delivers more clicks to that publisher, and is hence more desirable to it.
\item Publisher B: This is also a binary score (as above) corresponding to whether document is curated by Publisher~B.
\end{enumerate}

In our experiments, we implement Algorithm \ref{alg:traffic_shaping}, i.e. use a subset of the data for learning the dual prices (making arbitrary decisions in the process), and in the remaining sessions making the optimal decisions with respect to the learned dual prices. In our experiments, the online optimization problem involves:
(a) maximizing the dwell-time in each session,
(b) delivering a fixed number of clicks to Publishers A and B,
(c) achieving at least a certain average newsiness score over the entire set of sessions.

One of the key requirements of our framework is that we need the dwell-time, click probability, and newsiness score for each document when it is shown in any one of the $20$ slots. The machine learned models generate these scores with respect to the first slot only. In order to address this, we utilize the (somewhat standard) Reference CTR model, which is explained below.
\subsection{The Reference CTR}
When users are presented with a list of content on a web-page, it is well-known \cite{Chapelle} that there is a strong \emph{position effect} for clicks, dwell-time, and other engagement metrics (such as the newsiness score of interest to us in this paper). It is commonly assumed \cite{Chapelle}, that a document $d$, when shown in slot $p$, has a CTR that depends on its CTR in position $1$ (denoted by $c_1(d)$) and a factor that depends only on the position (denoted by $\text{ref}_p$). It is normalized so that $\text{ref}_1=1$. Hence the predicted CTR of document $d$ in position $p$ is given by:
%
%
$$
c_p(d)=c_1(d)\times\text{ref}_p.
$$

The quantity $\text{ref}_p$ is computed by computing the aggregate CTR for each position, and normalizing by the aggregate CTR for position $1$ \cite{Chapelle}. We will also assume (for this experiment) that the dwell-time and newsiness have a similar position effect as above.

For a list of items presented to a user according to a permutation $\sigma$, the total number of (predicted) clicks generated by the session, and the clicks generated for publisher $t$ (where $t \in \left\{A,B\right\})$ are respectively given by:
\begin{align*}
C(\sigma)&=\sum_{p=1}^{m}c_p(\sigma(p))\times\text{ref}_p \\
C_t(\sigma)&=\sum_{p=1}^{m}c_p(\sigma(p))\times\text{ref}_p \times \mathbf{1}\left(\sigma(p) \text{ of traffic type } t \right).
\end{align*}

Similarly, we assume that the total dwell-time and newsiness respectively are:
$$
D(\sigma)=\sum_{p=1}^{m}\texttt{dwell}(\sigma(p))\times\text{ref}_p,
\qquad 
N(\sigma)=\sum_{p=1}^{m}\texttt{news}(\sigma(p))\times\text{ref}_p.
$$

Figure \ref{fig:refctr} shows the reference CTR distribution. 

\begin{figure}[!htb]
\centering
\includegraphics[scale=.3]{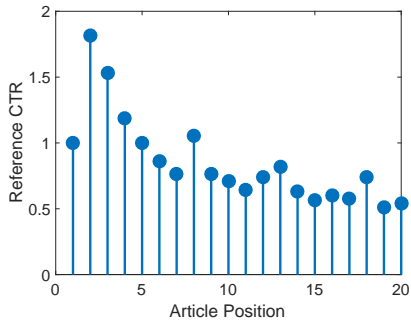}
\caption{The reference CTR distribution for the slots on the page. Interestingly, the reference CTR distribution is not monotonically decreasing with the slot position due to various page presentation effects.}
\label{fig:refctr}
\end{figure}

\subsection{Observations}
We first investigate the performance of sampled linear program when the size of the sample set increases to $2000$ sessions (entire dataset). Figure \ref{fig:costVSsessions} shows how the average (normalized) dwell-time per session varies as we increase the size of the sampled linear program. Note that when all $2000$ sessions are used, we essentially obtain the hindsight-optimal solution. As Fig. \ref{fig:costVSsessions} shows, this level of performance is approximately achieved by the sampled linear program withing just $800$ sessions, justifying our intuition that this quantity is distributionally stable with respect to the traffic.

\begin{figure*}
\begin{tabular}{cc}
\begin{subfigure}{0.45\textwidth}\centering\includegraphics[scale=.32]{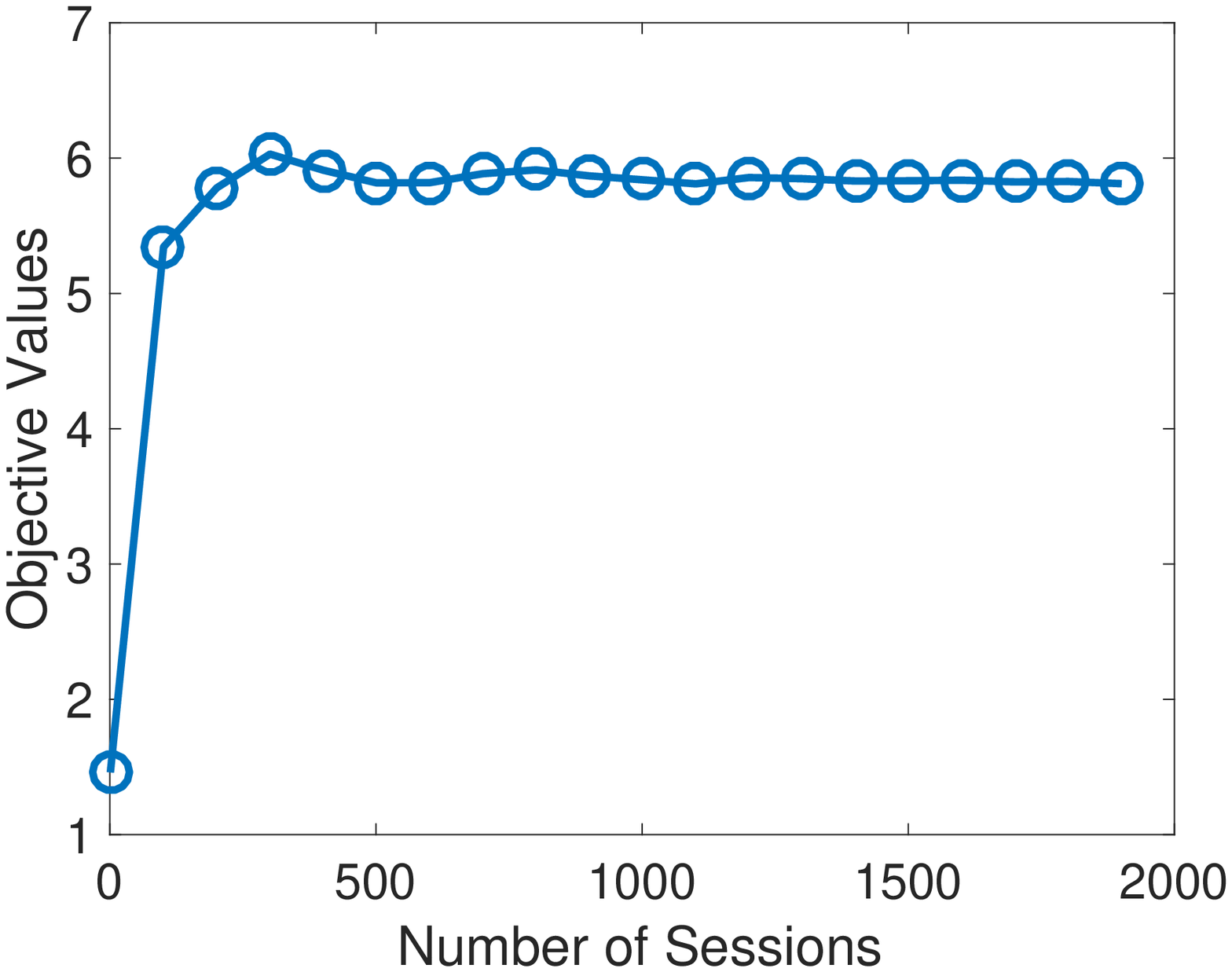}\caption{}\label{fig:costVSsessions}{}\end{subfigure} &
\begin{subfigure}{0.45\textwidth}\centering\includegraphics[scale=.32]{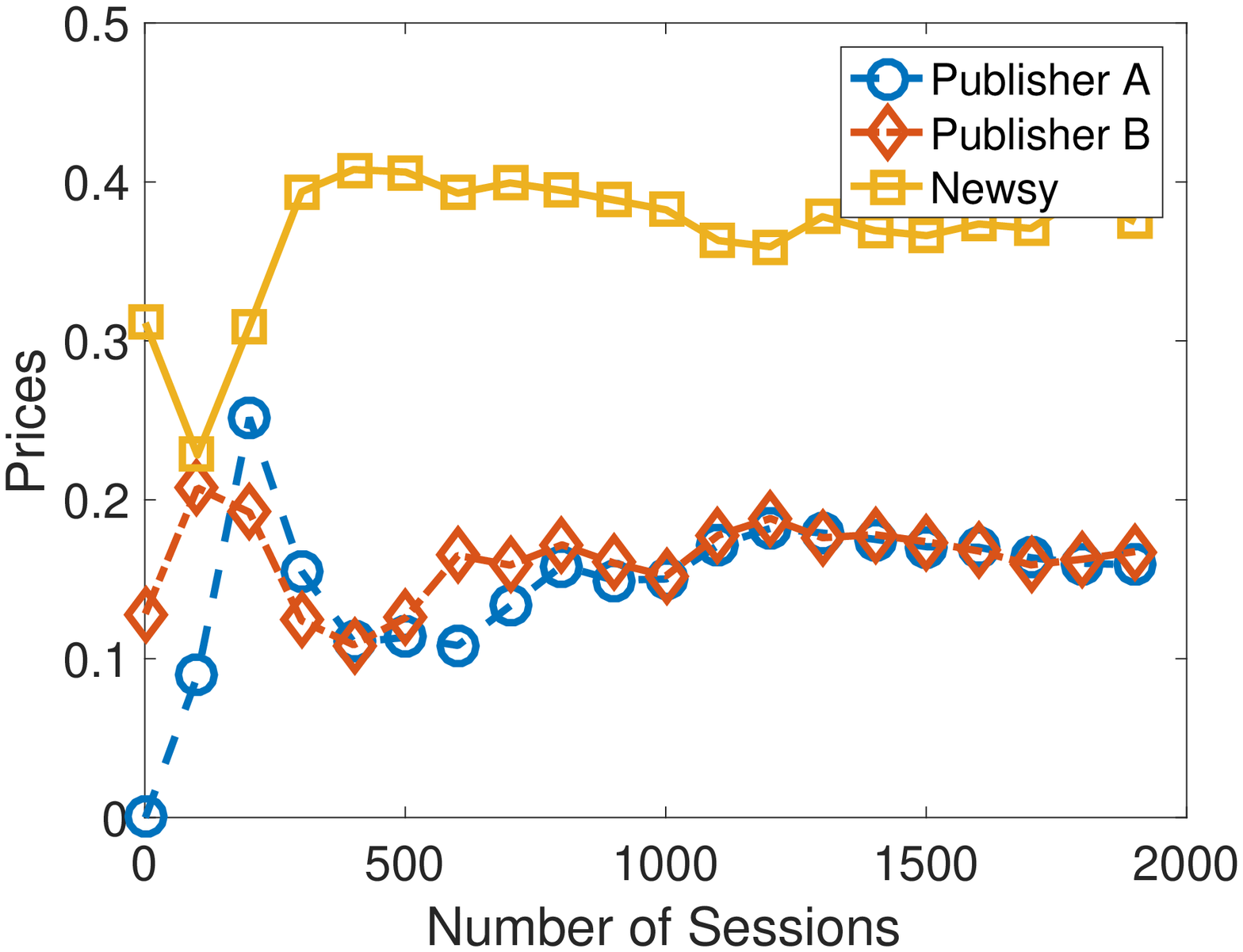}\caption{}\label{fig:pricesVSsessions}\end{subfigure} 
\end{tabular}
\caption{(a) Average Dwell-time per session as a function of the sample size for the sampled linear program. The dwell-time per session stabilizes at approximately $800$ sessions and is constant thereafter. (b) Dual prices for each of the constraints, as a function of the sample size for the sampled linear program. The dual prices also stabilize at approximately $800$ sessions and are constant thereafter. }
\label{fig:main}
\end{figure*}

\begin{figure*}
\begin{tabular}{cc}
\begin{subfigure}{0.45\textwidth}\centering\includegraphics[scale=.32]{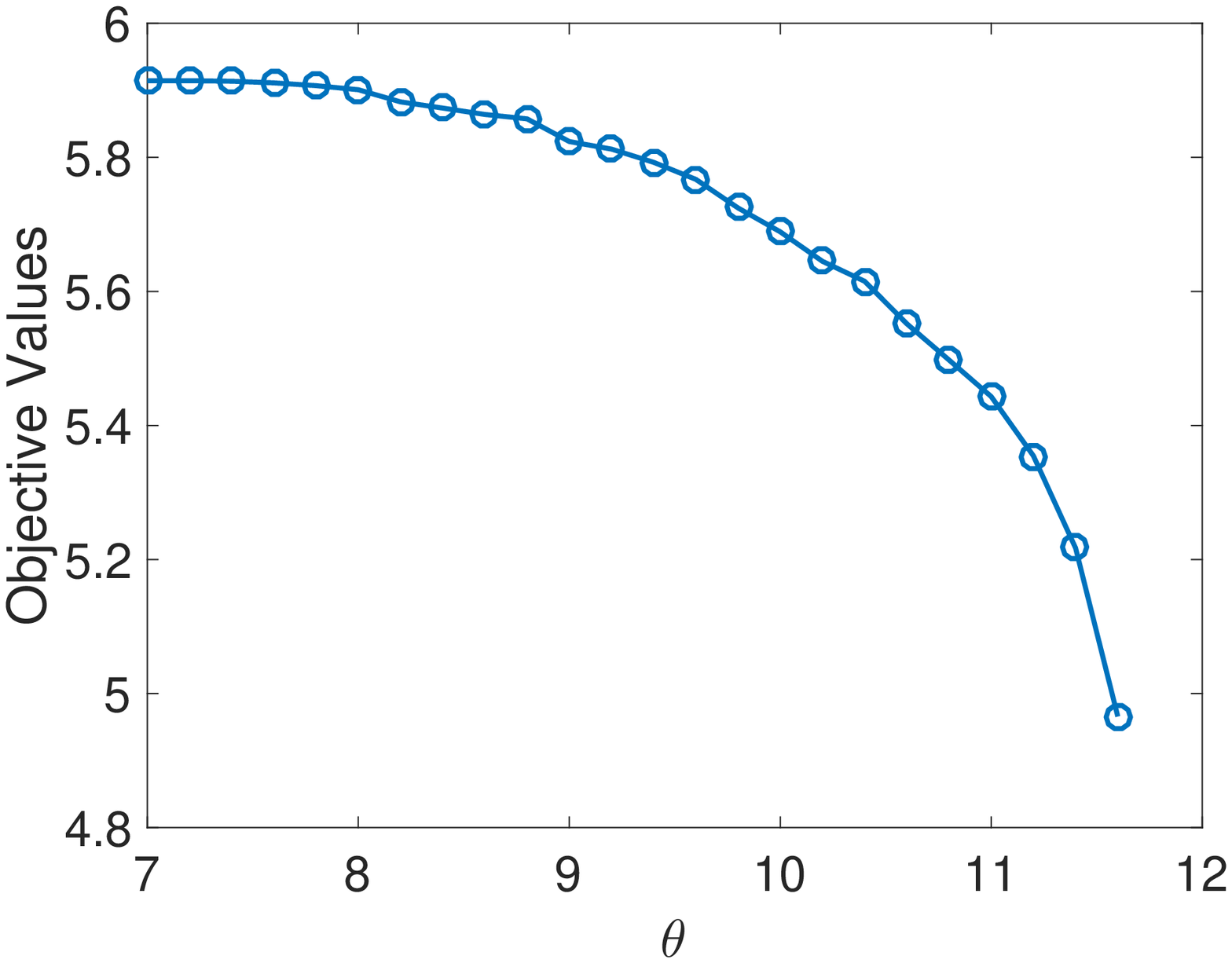}\caption{}\label{fig:costVSconstraint}\end{subfigure} &
\begin{subfigure}{0.45\textwidth}\centering\includegraphics[scale=.32]{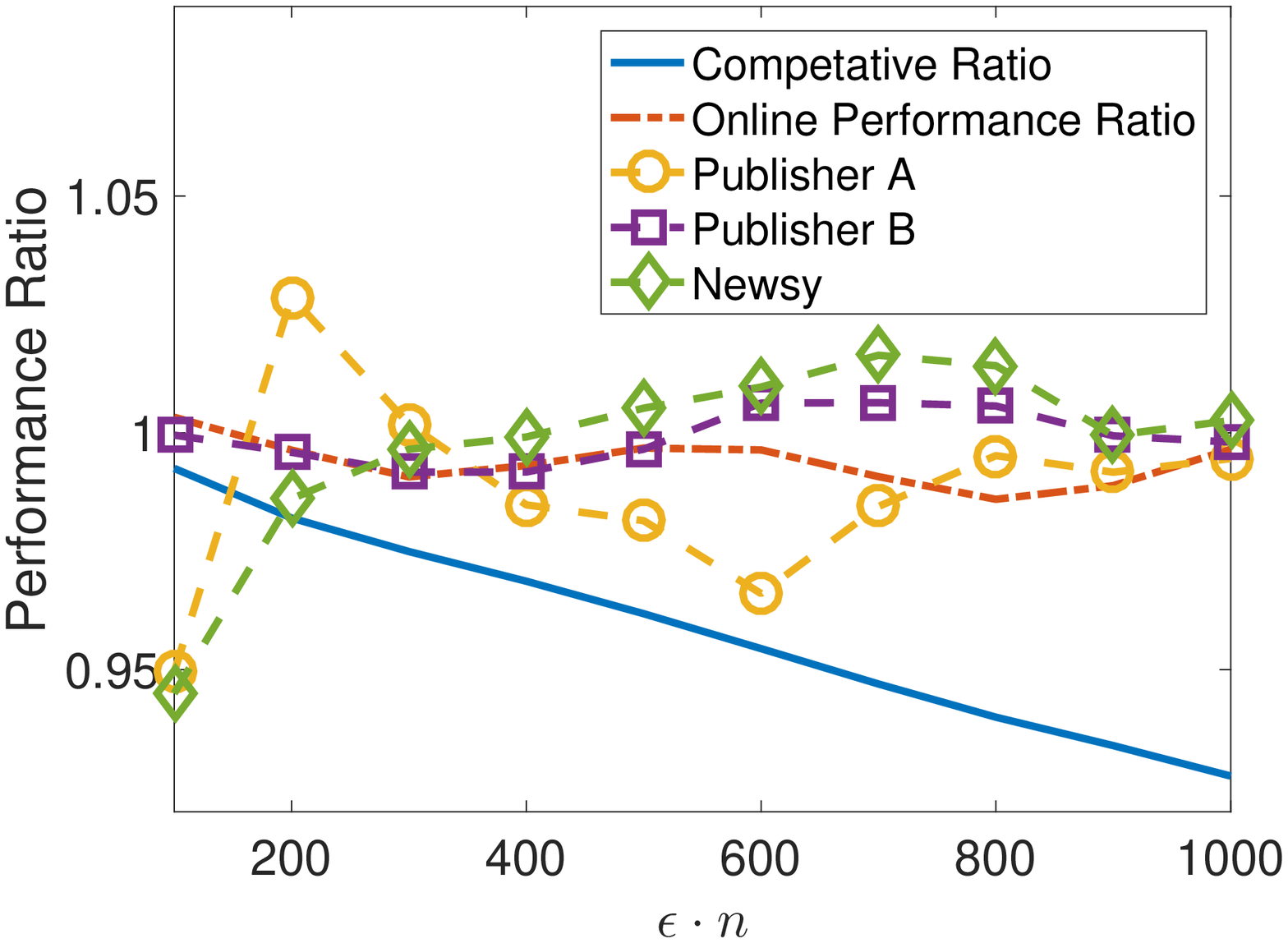}\caption{}\label{fig:performanceVSepsilon}\end{subfigure}  
\end{tabular}
\caption{ (a) Trade-off between dwell-time per session and clicks delivered to Publisher A. (b) Different performance ratios as a function of $\epsilon$.}
\label{fig:main}
\end{figure*}


Figure \ref{fig:pricesVSsessions} shows how the dual prices change as a function of the number of sessions over which the sampled linear program is solved. Again, we see that the prices stabilize at approximately $800$ sessions.

In Fig. \ref{fig:costVSconstraint} we investigate the trade-offs that are inevitable as a consequence of traffic-shaping. We fix the number of sessions to be $2000$ and solve the corresponding linear program (i.e. the hindsight optimal solution). We fix the values of the constraints for clicks delivered to publisher $B$ and the newsiness constraint. We vary the number of clicks deliverable to publisher $A$ by changing a parameter $\theta$ (as $\theta$ increases, the corresponding click constraint commitment $b_t$ increases multiplicatively). We study how the dwell-time per session changes as $\theta$ increases --- the corresponding trade-off curve is shown in Fig. \ref{fig:costVSconstraint}.

In Fig. \ref{fig:performanceVSepsilon} we study the performance of the online algorithm. On the $x$-axis is the fraction of the $2000$ sessions which were used for the sampled linear program. On the $y$-axis is the performance ratio for different metrics. A value of $\nu=1.05$ (as required in \eqref{opt:D2}) was chosen for this experiment. The solid-blue curve is the competitive ratio (i.e. the ratio of the performance by the online algorithm to that of the hindsight optimal solution, see Sec. \ref{sec:theory}). Note that in the first $\epsilon$ fraction of the sessions (the price-learning phase), arbitrary suboptimal decisions are made, and hence the competitive ratio decreases as a function of $\epsilon$ (see Theorem \ref{thm:main}). 
The yellow, purple, and green plots are the ratio of the clicks/newsiness delivered versus the value committed by the constraints (i.e. the $b_t$), and values under one indicate constraint violation. Note that for small values of $\epsilon$, some of the constraints are actually violated --- hence while the competitive ratio is high, training on that small subset of sessions yields infeasible solutions.
The red dotted line indicates the online performance ratio as a function of $\epsilon$. The online performance ratio is the ratio of the reward of the online algorithm to that of the hindsight optimal sessions, \emph{restricted to the segment of the sessions where online decisions are made}. While the online performance ratio is close to $1$ throughout, it attains this ratio by violating constraints at small values of $\epsilon$.

Another interesting feature that we found was that, under the RefCTR model, the greedy algorithm (i.e. sorting documents by the price-weighted scores, and then matching them to positions in sorted order of their RefCTR value) achieve the \emph{same ranking} compared to the Hungarian algorithm. This is simply a structural consequence of the RefCTR model (we omit the proof here), however the scalability implications are substantial; instead of implementing the Hungarian algorithm, the serving implementation only needs a sorting procedure.


\section{Theoretical Guarantees} \label{sec:theory}
In order to describe the performance guarantees, we need to define a critical (and well-known) solution concept related to online algorithms, namely the notion of the competitive ratio of the algorithm.

Given online events $\left\{ (C_k, A_{k1}, \ldots, A_{km})\right\}$ for sessions $k \in [n]$, let OPT denote the solution obtained by solving the linear program \eqref{opt:P1} on this data --- we call the solution thus obtained the \emph{optimal hindsight solution}. A hypothetical decision-maker with advance knowledge of the ``future'' sessions would make these optimal decisions, and this is the best performance achievable by any algorithm. Suppose Algorithm \ref{alg:traffic_shaping} (which is online) achieves a performance of $\widehat{\text{OPT}}$. We define the \emph{competitive ratio} of the algorithm to be $\frac{\widehat{\text{OPT}}}{\text{OPT}}$. The competitive ratio measures how the performance of the online algorithm compares to the optimal hindsight solution. Our main result shows that, under certain assumptions, our algorithm achieves an $1-O\left( \epsilon \right)$ competitive ratio. We state our assumptions next:

\begin{enumerate} [leftmargin=*]
\item \textbf{Assumption 1:} The sessions $\left\{ (C_k, A_{k1}, \ldots, A_{kT})\right\}$  arrive in a random order, i.e. the matrices  $\left\{ (C_k, A_{k1}, \ldots, A_{kT})\right\}$ can be arbitrary, but every permutation of the indices $k$ have an equal probability.
\item \textbf{Assumption 2:} The optimization problems \eqref{opt:P1} and \eqref{opt:P2} are both strictly feasible.
\item \textbf{Assumption 3:} The total number of sessions $n$ is known a priori.
\item \textbf{Assumption 4:} For every permutation $P$ we have $\langle C_k, P \rangle \in [0,1]$ and  $\langle A_{kt}, P \rangle \in [0,1]$ for all $k=1, \ldots, n$ and $t=1, \ldots, T$.
\end{enumerate}

\begin{theorem} \label{thm:main}
Consider the online traffic-shaping problem when the above assumptions hold true, and let $\epsilon \in (0,\frac{1}{3})$ be fixed. Define $B:=\min_{t=1, \ldots, T} b_t$. Suppose the constraints satisfy the requirement that
\begin{equation} \label{eq:sample_complexity}
\epsilon \geq \left( \frac{\log \left( \frac{2C_0T}{\epsilon} \right) + (T+1) \log \left( m^2n \right)}{B} \right)^{\frac{1}{3}}.
\end{equation}

Let the dual prices computed by the sampled linear program be denoted by $\hat{\lambda}$.
Then the solution constructed by Algorithm \ref{alg:traffic_shaping}, denoted by $P_k(\hat{\lambda})$, has the following properties:
\begin{enumerate}
\item In each session the optimal hindsight solution and the online approach via Algorithm \ref{alg:traffic_shaping} produce  valid assignments $P_k^*$ and $P_k^{\text{online}}(\hat{\lambda})$. Moreover, if the optimal hindsight prices, denoted $\lambda^*$, are given \ref{alg:traffic_shaping}, then   $P_k^* = P_k^{\text{online}}(\lambda^*)$.
\item When $\nu = 1 + 4 \epsilon$, With probability exceeding $1 - \epsilon$, the assignments satisfy the traffic-shaping constraints, i.e. 
$$
\sum_{k=1}^{n} \langle A_{kt}, P_k(\hat{\lambda}) \rangle \geq b_t \; \; \forall t=1, \ldots, T
$$
\item When $\nu = 1 - \epsilon$, with probability exceeding $1-\epsilon$ the online solution satisfies
$$
\sum_{k=1}^{n} \langle A_{kt}, P_k(\hat{\lambda}) \rangle \geq (1-2\epsilon)b_t \; \; \forall t=1, \ldots, T
$$
and objective attained by the online assignments $P_k(\hat{\lambda})$ satisfies:
$$
\sum_{k=1}^{n} \langle C_k, P_k(\hat{\lambda}) \rangle \geq (1-\epsilon) \text{OPT},
$$
(i.e. the competitive ratio is $1-O(\epsilon)$ with high probability).
\end{enumerate}
\end{theorem}
We give a detailed proof in Sec \ref{sec:proof}. 
We make the following remarks about the main result:\\
\textbf{Remarks \\}
(1) Note that the first part states that if the optimal prices are available to the online algorithm, then optimal decisions will be made in each session. Hence the prices $\lambda^*$ are sufficient to make good decisions online. The second and third parts of the theorem quantify the intuition that if we use a small set of samples initially to learn the prices using the sampled linear program, the online algorithm will then perform almost optimally.

\noindent (2) Note that Algorithm \ref{alg:traffic_shaping} advocates using the Hungarian algorithm to compute the best matching in each session. Depending on the number of items to be ranked, this may or may not be feasible at run time --- in the latter case a greedy algorithm may be used. The greedy algorithm, being a $\frac{1}{2}$-approximation in the worst case, can also be analyzed and bounds on the competitive ratio may be derived. In our experiments, we find that using the greedy algorithm does not lead to any loss of performance as a consequence of the RefCTR assumption (see Sec. \ref{sec:expts}).

\noindent (3) Assumption $4$ requires that the for each permutation, the rewards and the influence to each traffic shaping is non-negative and bounded by $1$. The non-negativity assumption is natural in our setting (since we are interested in quantities such as views and clicks). The upper bound of $1$ is somewhat arbitrary, we merely need the rewards and influences to be bounded by some uniform quantity (say $U$). Our results would be more or less unaffected (modulo certain logarithmic factors in the statement of Theorem \ref{thm:main}) if $U > 1$).

\noindent (4) We now comment on the qualitative relationship between $\epsilon$, $B$ and the competitive ratio. In order to make sense of the statement, we assume that $B$ increases linearly with the number of sessions $n$ (e.g. clicks committed increase linearly as the number of sessions increases); i.e. $B=\rho n$. (Note that $B$ growing with $n$ fixed would be problematic from a feasibility point of view - if we over-commit to a traffic-shaping constraint the optimization problem becomes infeasible violating Assumption $2$). Thus as $n$ increases, the right hand side of \eqref{eq:sample_complexity} decreases to zero and we may make $\epsilon$ arbitrarily small, thus yielding (a) a small fraction of the samples used for learning the prices (b) a better competitive ratio, and (c) a high probability of success simultaneously. Indeed \eqref{eq:sample_complexity} may be viewed as a ``sample-complexity'' type result.

\noindent (5) The competitive ratio decreases as the fraction of sessions $\epsilon$ used for learning increases. While on the one hand the prices are learned more accurately as $\epsilon$ increases, on the other hand suboptimal decisions are made by the online algorithm in the learning phase. Indeed, we assume (somewhat pessimistically) that arbitrary decision are made in the learning phase with zero reward accrued, thus an $\epsilon$-fraction of the possible reward is lost in this phase. 

\section{Conclusion } \label{sec:conclusion}
In this paper we investigated the problem of online constrained ranking and its application to the web traffic-shaping problem. We developed a linear-programming based primal-dual algorithm for online ranking and demonstrated it's efficacy on real data. We also proved rigorous guarantees about the algorithm in terms of the achieved competitive ratio.

While the paper focuses on traffic-shaping as the motivating application, we believe this algorithm to have broad applicability to whole-page optimization and beyond. For instance, when optimizing the simultaneous placement of ads and organic content, additional constraints (or modifications to the objective) can be made to factor in the revenue considerations traffic-wide. More generally, this approach can be used for solving a variety of constrained online assignment problems such as matching customers to servers (e.g. matching drivers to passengers) subject to overall service-level constraints. We believe this approach will scale well to a variety of such application domains.

 A number of further avenues for future work are worth mentioning:
 \begin{itemize}
 \item Our model involves optimizing (in the objective) subject to constraints over the predicted dwell-time and the predicted clicks over a period of sessions. These predictions are obtained as the output of a machine-learned model. The number of actual clicks obtained, and the actual dwell-time realized will of course be different. As a practical strategy, we advocate re-solving the linear program periodically (and suitably altering the constraints) to capture these variations. However, there is an unavoidable feedback loop between the rankings and the machine learned models --- the decisions made in the ranking phase affect the learning of these online models. Consequently, there is a natural exploration-exploitation trade-off that is unavoidable in this setup. Ranking unknown items in a highly clickable slot will lead to more exploration, whereas ranking more certain items will lead to exploitation. How one navigates this will likely be of key practical importance.
 \item Our approach assumes that the objective and the constraints can be expressed as linear functions of the ranking; the expected number of clicks and expected dwell time are assumed to be a weighted sum of the components derived from each slot. Other types of constraints may not be linearly expressible (e.g. suppose one wishes to enforce a certain notion of diversity in each ranking and enforce a certain minimum amount of average diversity across sessions). Working with non-linear constraints will likely enable us to enrich the quality of ranking in our traffic. 
 \end{itemize}

\section{Proofs} \label{sec:proof}
We begin by stating two lemmas (their proofs are omitted due to space constraints) which follow as a standard consequence of the analysis of the Hungarian algorithm (we refer the reader to \cite{Hungarian} for details).

\begin{lemma} \label{lemma:hung}
Let $C \in \R^{m \times m}$ be an entry-wise non-negative matrix. Consider the following optimization problems:
\begin{equation}
\begin{split} \label{opt:P3}
\underset{P}{\text{maximize}} & \qquad  \langle C, P \rangle \\
\text{subject to}   
& \qquad P \in \mathcal{B}
\end{split}
\end{equation}

\begin{equation}
\begin{split} \label{opt:D3}
\underset{\alpha, \beta}{\text{maximize}} & \qquad \alpha^{\prime} \mathbf{1} + \beta^{\prime} \mathbf{1} \\
\text{subject to}  & \qquad C + \mathbf{1} \alpha^{\prime}+  \beta \mathbf{1}^{\prime} \leq 0 
\end{split}
\end{equation}
Then the following hold:
\begin{enumerate}
\item The pair \eqref{opt:P3}, \eqref{opt:D3} is a primal-dual pair, and strong duality holds.
\item An optimal solution $P^*$ to \eqref{opt:P3} is attained at a permutation matrix, and this solution can be found using the Hungarian algorithm.
\end{enumerate}
\end{lemma}
 \begin{proof}
 The duality between \eqref{opt:P3} and \eqref{opt:D3} is a standard exercise - we leave it to the reader. To see that strong duality holds, we note that the matrix with entries $P_{ij}=\frac{1}{m}$ is in the relative interior of the feasible set of \eqref{opt:P3}. By Slater's condition \cite[Chap. 5.2.3]{Boyd} strong duality follows.

 The second part follows from the fact that the permutation matrix corresponding to the maximum weight perfect matching is an optimal solution. To see this, first note that a perfect matching exists since the bipartite graph in question is complete (since $C \in \R^{m \times m}$ is square), and thus satisfies Hall's theorem \cite[Chap. 22]{Hungarian} - guaranteeing existence. Furthermore, the Hungarian algorithm \cite{Hungarian} yields a permutation matrix $P^*$ such that dual feasibility is achieved (see e.g. \cite[Chap. 17]{Hungarian}), indeed dual feasibility of \eqref{opt:D3} is one of the key invariants of the Hungarian algorithm. Furthermore, the corresponding permutation matrix $P^*$ satisfies complementary slackness, only edges $(i,j)$ corresponding to a perfect matching are chosen (primal feasibility) corresponding to edges where the dual constraint $\left( C + \mathbf{1} \alpha^{\prime}+ \mathbf{1} \beta^{\prime} \right)_{ij} = 0$ i.e. corresponding to tight constraints, and hence $\left( C + \mathbf{1} \alpha^{\prime}+ \mathbf{1} \beta^{\prime} \right)_{ij} P^*_{ij} = 0$. Hence $P^*$ must be a primal optimal solution.
 \end{proof}
\begin{lemma} \label{lemma:wt_hung}
Let $\lambda = \bar{\lambda} \in \R^{T}$ be fixed such that $\bar{\lambda} \geq 0$. Consider the solution to the optimization problem 
\begin{equation}
\begin{split} \label{opt:D4}
\underset{\alpha_k, \beta_k}{\text{maximize}} & \qquad \sum_{t=1}^T \bar{\lambda}_t b_t + \sum_{k=1}^n \alpha_k^{\prime} \mathbf{1} + \sum_{k=1}^n \beta_k^{\prime} \mathbf{1} \\
\text{subject to}  & \qquad C_k + \sum_{t=1}^{T} \lambda_t A_{kt} + \mathbf{1} \alpha_k^{\prime}+  \beta_k \mathbf{1}^{\prime} \leq 0 \; \; \forall k=1, \ldots, n 
\end{split}
\end{equation}
A set of optimal dual solutions (w.r.t. the constraints of \eqref{opt:D4}) $P_k(\bar{\lambda})$ for $k=1, \ldots, n$ corresponding to \eqref{opt:P1} are achieved at extreme points of the Birkhoff polytope, (i.e. permutation matrices). Moreover each of the $P_k(\bar{\lambda})$ may be computed by computing a maximum-weight perfect matching with respect to the bipartite graph with weights $C_k + \sum_{t=1}^T \bar{\lambda}_t  A_{kt}$ (using the Hungarian algorithm).
\end{lemma}
 \begin{proof}
 Note that for each fixed $\lambda = \bar{\lambda}$ and the separability (with respect to $k$) of the objective function, \eqref{opt:D4} decouples into the following optimization problems:
 \begin{equation}
 \begin{split} \label{opt:D6}
 \underset{\alpha_k, \beta_k}{\text{maximize}} & \qquad   \alpha_k^{\prime} \mathbf{1} +  \beta_k^{\prime} \mathbf{1} \\
 \text{subject to}  & \qquad C_k + \sum_{t=1}^{T} \bar{\lambda}_t A_{kt} + \mathbf{1} \alpha_k^{\prime}+  \beta_k \mathbf{1}^{\prime} \leq 0, 
 \end{split}
 \end{equation}
 and the corresponding optimal solutions $\alpha_k^*(\bar{\lambda})$, $\beta_k^*(\bar{\lambda})$ are optimal with respect to \eqref{opt:D4}.  The strong dual of \eqref{opt:D6}  is precisely:
 \begin{equation}
 \begin{split} \label{opt:P3}
 \underset{P}{\text{maximize}} & \qquad  \langle C_k + \sum_{t=1}^{T} \bar{\lambda}_t A_{kt} , P \rangle \\
 \text{subject to}   
 & \qquad P^{\prime} \mathbf{1} =\mathbf{1} \\
 & \qquad P \geq 0 \; \; k=1, \ldots, n
 \end{split}
 \end{equation}
 By Lemma \ref{lemma:hung},  the solutions $P_k(\bar{\lambda})$ to the above are yielded by the permutations corresponding to the maximum-weight perfect matchings yielded by running the Hungarian algorithm on bipartite graphs with weight matrices $C_k + \sum_{t=1}^{T} \bar{\lambda}_t A_{kt}$ for $k=1, \ldots, T$.
\end{proof}
We will also need the following lemma concerning concentration of random variables. For a proof we refer the reader to \cite[Lemma A.1]{Ye}, \cite[Lemma 3]{Devanur} and the references therein.
\begin{lemma} \label{lemma:conc}
Let $u_1, \ldots, u_r$ be a uniformly random sample without replacement from real numbers  $$\left\{ | \sum_{i=1}^r u_i - r \bar{c} | \geq t \right\} \leq \exp\left( -\frac{t^2}{2r\sigma_R^2 + t \Delta_R} \right),$$
where $\Delta_R= \max_i c_i - \min_i c_i$, $\bar{c}=\frac{1}{R} \sum_{i=1}^R c_i$, and $\sigma_R^2=\frac{1}{R}\sum_{i=1}^R(c_i -\bar{c})^2$.
\end{lemma}
 We now provide a detailed proof of Theorem 1.
Note that both our main result for online ranking, and our assumptions are of a similar flavor as existing literature in online algorithms \cite{Devanur,Ye}. Indeed, it is known \cite{Devanur,Ye}, that in these related results, these assumptions are also tight, i.e. it is not possible to get $1-O\left( \epsilon \right)$ competitive ratio if a single one of these assumptions is removed. While we do not show the tightness of the assumptions, we believe it likely that they are needed to achieve the stated competitive ratio. 

\begin{proof}[Proof of Theorem \ref{thm:main}] 
The main idea behind the proof, which mimics \cite{Ye} with suitable adaptations, will be to use the complementary slackness conditions for linear programming. Let $\hat{\lambda}$ be a dual price vector obtained as the solution of the sampled linear program. We denote by $P_k(\hat{\lambda})$ the permutation that is chosen as per steps \ref{step:scoring}, \ref{step:hungarian} of the algorithm. 
\textbf{ \\ \\ Part 1.} 
The online assignments $P_k^{\text{online}}(\hat{\lambda})$ (and similarly $P_k^{\text{online}}(\lambda^*)$) are valid assignments since they are constructed via the Hungarian algorithm. Note that they are both trivially primal-feasible (since assignments are extreme points of the Birkhoff polytope).\\ 

Let $\lambda^*$ be an optimal dual solution of \eqref{opt:D1}. Note that when $\bar{\lambda} = \lambda^*$, the corresponding optimal solutions of \eqref{opt:D4} coincide with those of \eqref{opt:D1}. By strong duality, a set of corresponding primal solutions also coincide. By Lemma \ref{lemma:wt_hung}, a dual variable corresponding to the constraints of \eqref{opt:D4} achieves optimality at permutation matrices $P_k(\lambda^*)$. Hence, $P_k(\lambda^*)=P_k^*$. 
\\
\textbf{Part 2.} Our proof strategy will be as follows. We will take a hypothetical optimal solution $\hat{\lambda}$ to the sampled linear program. It will be assumed to satisfy that the complementary slackness conditions of the primal dual pair \eqref{opt:P3},\eqref{opt:D3} - but otherwise arbitrary. We will show that as a consequence of complementary slackness and the assumptions, with high probability the solution yielded by $P_k(\hat{\lambda})$ are feasible for this hypothetical $\hat{\lambda}$. We will then take a union bound over all possible $\hat{\lambda}$ to obtain that the solution is feasible unconditionally. 

Let us fix a constraint $t$ and note that for the first $\epsilon n$ sessions, arbitrary decisions are made and we thus assume that no contributions to the constraints are made from these. As a result we required that the sampled linear program satisfy:
$$
\sum_{k \in S} \langle A_{kt}, P_k \rangle \geq (1+4 \epsilon) \epsilon b_t,
$$
where the additional $(1+4 \epsilon)$ factor ensures that the solution robustly satisfies the constraints.

The bad event of interest to us is the event that a constraint is violated. We will call the set $S$ bad when the solution $\lambda$ leads us to a situation of such a constraint violation. Concretely, we consider the following set relations corresponding to a set $S$ being bad:

\begin{align*}
&\left\{\sum_{k \in N \setminus S }  \langle A_{kt}, P_k(\lambda) \rangle < b_t, \sum_{k \in S}  \langle A_{kt}, P_k(\lambda) \rangle \geq (1+4 \epsilon)\epsilon b_t \right\} \subseteq A_1 \cup A_2, \\
&A_1 = \left\{\sum_{k \in N \setminus S }  \langle A_{kt}, P_k(\lambda) \rangle < b_t, \sum_{t \in S}  \langle A_{kt}, P_k(\lambda) \rangle \geq (1+4 \epsilon)\epsilon b_t, \right. \\ 
& \qquad \left. \sum_{k \in N}  \langle A_{kt}, P_k(\lambda) \rangle \geq (1+3 \epsilon) b_t  \right\} \\
&A_2 = \left\{ \sum_{k \in S}  \langle A_{kt}, P_k(\lambda) \rangle \geq (1+4 \epsilon)\epsilon b_t, \sum_{k \in N}  \langle A_{kt}, P_k(\lambda) \rangle < (1+3 \epsilon) b_t  \right\}
\end{align*}

Next we note that:
\begin{align*} 
\Q(A_1) & \leq \Q \left(\sum_{k \in N \setminus S }  \langle A_{kt}, P_k(\lambda) \rangle < b_t, \sum_{k \in N}  \langle A_{kt}, P_k(\lambda) \rangle \geq (1+3 \epsilon) b_t \right) \\
& \hspace{-4mm} \leq \Q \left(\sum_{k \in N \setminus S }  \langle A_{kt}, P_k(\lambda) \rangle < b_t, \sum_{k \in N}  \langle A_{kt}, P_k(\lambda) \rangle \geq \frac{1+\epsilon}{1-\epsilon} b_t \right) \\
&\hspace{-4mm} =   \Q \left(\sum_{k \in N \setminus S }  \langle A_{kt}, P_k(\lambda) \rangle < b_t, (1-\epsilon)\sum_{k \in N}  \langle A_{kt}, P_k(\lambda) \rangle \geq (1+\epsilon) b_t \right) \\
&\hspace{-4mm} \leq 2 \exp \left( -\frac{\epsilon^2 b_t}{2-\epsilon }\right)
\end{align*}
where the second inequality follows since $\frac{1+\epsilon}{1-\epsilon} < 1+3 \epsilon$ for all $\epsilon \in (0,\frac{1}{3})$, and the last inequality follows from Lemma \ref{lemma:conc}.

Defining $Y_k:=\langle A_{kt}, P_k(\lambda) \rangle$, we also note that \small 
\begin{align*}
\Q(A_2) & \leq \Q \left( \sum_{t \in S}  \langle A_{kt}, P_k(\lambda) \rangle \geq (1+4 \epsilon)\epsilon b_t, \sum_{t \in N}  \langle A_{kt}, P_k(\lambda) \rangle < (1+3 \epsilon) b_t  \right) \\
& \leq \Q \left( |\sum_{t \in S} Y_t - \epsilon \sum_{t \in N} Y_t| \geq \epsilon^2 b_t\right) \\
& \leq 2 \exp \left( -\frac{\epsilon^3 b_t}{2+ \epsilon}\right)
\end{align*} \normalsize
where the last inequality again follows from an application of Lemma \ref{lemma:conc}.

Hence, the probability of a bad sample is bounded above by $4 \exp \left( -\frac{\epsilon^3 b_t}{2+ \epsilon} \right) \leq \delta$, where $\delta:= \frac{\epsilon}{C_0 (m^2n)^{T+1}} $




Since $\lambda$ is the solution obtained from the sampled linear program, it follows that:
$$
\sum_{k \in S} Y_k \geq (1+\epsilon) \epsilon b_t.
$$

where $\delta := \frac{\epsilon}{2C_0T(m^2n)^{T+2}}$ for some constant $C_0$. In the preceding chain of inequalities, the last one follows from Lemma \ref{lemma:conc}.

Since the solution $\lambda$ obtained from the sampled linear program can be arbitrary, we take a union bound over all the possible ``distinct'' $\lambda$.  We call two solutions $\lambda_1$ and $\lambda_2$ as distinct if $P_k(\lambda_1) \neq P_k(\lambda_2)$ for some session $k \in N$. Consider a dual constraint with complementary slackness: 
$$
\left( C_k + \sum_{t=1}^{T} \lambda_t A_{kt} + \mathbf{1} \alpha_k^{\prime}+ \mathbf{1} \beta_k^{\prime} \right)_{ij} \left(P_k \right)_{ij}=0.
$$
and note that $\left( P_k \right)_{ij} > 0$ is only possible when $$\left( C_k + \sum_{t=1}^{T} \lambda_t A_{kt} + \mathbf{1} \alpha_k^{\prime}+ \mathbf{1} \beta_k^{\prime} \right)_{ij} \geq 0.$$ (Indeed, the set of possible solutions yielded by the Hungarian algorithm will be a subset of such possible $P_k$ - by design the Hungarian algorithm maintains the dual inequality as an invariant \cite{Hungarian}). The set of possible $P_k$ is thus bounded above by the number of unique separations of the points $$\left\{ \left( \left( C_k  \right)_{ij},  \left( A_{k1},  \right)_{ij}, \ldots, \left( A_{kT},  \right)_{ij},1 \right) \right\}_{i,j,k \in [m] \times [m] \times [n]}$$ in $(T+1)$-dimensional space. By results from computational geometry \cite{comp_geom}, the number of such distinct hyperplanes, is bounded above by $C_0 (m^2n)^{T+1}$ for some constant $C_0$. Applying a union bound over the possible different prices, and $t=1, \ldots, T$ we obtain the required result.
\\ \\
\textbf{Part 3.} 
Given $\nu=1-\epsilon$, one can prove approximate feasibility in a similar manner as above by observing that \small
$$
\Q \left( \sum_{k \in N \setminus S }  \langle A_{kt}, P_k(\lambda) \rangle <(1-2\epsilon) b_t, \sum_{k \in N}  \langle A_{kt}, P_k(\lambda) \rangle \geq (1 - \epsilon) b_t \right) \leq \delta,
$$ \normalsize
where $\delta$ is the quantity used in the preceding part.

To bound the competitive ratio, our proof strategy will be to construct an auxiliary linear program for which the online solutions are optimal. We will show that a point $P_k^*$ (the solution to \eqref{opt:P1}) is feasible with respect to this auxiliary linear program, and hence the objective function values between the online solution and that achieved by $P_k^*$ cannot be too different.

Let $\hat{\lambda}$ be the solution of the sampled linear program \eqref{opt:D2} and $P_k(\hat{\lambda})$ the corresponding assignment. Consider the linear program:
\begin{equation}
\begin{split} \label{opt:P_aux}
\underset{P_1, \ldots, P_{n}}{\text{maximize}} & \qquad \sum_{k=1}^{n} \langle C_k, P_k \rangle \\
 \text{subject to}  & \qquad \sum_{k=1}^{n} \langle A_{kt}, P_k \rangle \geq  \hat{b}_t \; \; \forall \; t=1, \ldots, T\\
 & \qquad P_k \in \mathcal{B} \; \; k=1, \ldots, n
\end{split}
 \end{equation}
where the quantity $\hat{b}_t$ is defined as $\hat{b}_t = \sum_{t \in N} \langle A_{kt}, P_k(\hat{\lambda})\rangle$ if $\hat{\lambda}_t>0$ and $b_t=\min \left\{ \sum_{t \in N} \langle A_{kt}, P_k(\hat{\lambda})\rangle, b_t\right\}$ if $\lambda_t =0$.
Since $\hat{\lambda},  P_k(\hat{\lambda})$  satisfy the complementary slackness conditions for \eqref{opt:P_aux} by construction, they are optimal solutions to it. Note that in general the solution of the sampled linear program could be fractional (i.e. $\hat{P}_k$ for some members of the Birkhoff polytope). However, as a consequence of Lemma \ref{lemma:wt_hung}, it is in fact the case that $\hat{P}_k=P_k(\hat{\lambda})$ are optimal.

Suppose $\hat{\lambda}_t >0$, then the $t^{th}$ primal constraint is tight in the sampled linear program, i.e. $\sum_{k \in S}  \langle A_{kt}, P_k(\hat{\lambda})\rangle = (1-\epsilon)\epsilon b_t$. 
By applying an argument along similar lines to part $2$ above and applying concentration, it follows that with probability at least $1-\epsilon$,
$$
\hat{b}_t = \sum_{t \in N} \langle A_{kt}, P_k(\hat{\lambda})\rangle \leq  b_t.
$$

Let $P_k^*$ be the optimal solution to \eqref{opt:P2}. It follows that $P_k^*$ are also feasible solutions to \eqref{opt:P_aux}. As a consequence, $\sum_{k \in N}\langle C_k, P_k(\hat{\lambda}) \rangle \geq \text{OPT}$, where OPT refers to the objective function value of the optimal solution to \eqref{opt:P2}.

By again applying a concentration argument as above, we can argue that with probability exceeding $1-\epsilon$,
$$
\sum_{k \in N \setminus S}\langle C_k, P_k(\hat{\lambda}) \rangle \geq (1-\epsilon)\sum_{k \in N }\langle C_k, P_k(\hat{\lambda}) \geq (1-\epsilon)\text{OPT}.
$$
Since the returns from the first $\hat{n}$ sessions are non-negative (under an arbitrary assignment), the online algorithm returns a reward of at least $(1-\epsilon)\text{OPT}$.
\end{proof}
\bibliographystyle{plain}
\bibliography{sigproc} 

\begin{thebibliography}{10}

\bibitem{Agarwal_cs3}
Deepak Agarwal, Shaunak Chatterjee, Yang Yang, and Liang Zhang.
\newblock Constrained optimization for homepage relevance.
\newblock In {\em Proceedings of the 24th International Conference on World
  Wide Web}, WWW '15 Companion, pages 375--384, New York, NY, USA, 2015. ACM.

\bibitem{Agarwal_cs2}
Deepak Agarwal, Bee-Chung Chen, Pradheep Elango, and Xuanhui Wang.
\newblock Click shaping to optimize multiple objectives.
\newblock In {\em Proceedings of the 17th ACM SIGKDD International Conference
  on Knowledge Discovery and Data Mining}, KDD '11, pages 132--140, New York,
  NY, USA, 2011. ACM.

\bibitem{Agarwal_cs1}
Deepak Agarwal, Bee-Chung Chen, Pradheep Elango, and Xuanhui Wang.
\newblock Personalized click shaping through lagrangian duality for online
  recommendation.
\newblock In {\em Proceedings of the 35th International ACM SIGIR Conference on
  Research and Development in Information Retrieval}, SIGIR '12, pages
  485--494, New York, NY, USA, 2012. ACM.

\bibitem{Ye}
Shipra Agrawal, Zizhuo Wang, and Yinyu Ye.
\newblock A dynamic near-optimal algorithm for online linear programming.
\newblock {\em Operations Research}, 62(4):876--890, 2014.

\bibitem{comb_auc}
M.~F. Balcan, A.~Blum, J.~D. Hartline, and Y.~Mansour.
\newblock Mechanism design via machine learning.
\newblock In {\em 46th Annual IEEE Symposium on Foundations of Computer Science
  (FOCS'05)}, pages 605--614, Oct 2005.

\bibitem{Boyd}
Stephen Boyd and Lieven Vandenberghe.
\newblock {\em Convex Optimization}.
\newblock Cambridge University Press, New York, NY, USA, 2004.

\bibitem{Bubeck}
S{\'{e}}bastien Bubeck.
\newblock Convex optimization: Algorithms and complexity.
\newblock {\em Foundations and Trends in Machine Learning}, 8(3-4):231--357,
  2015.

\bibitem{Buchbinder}
Niv Buchbinder and Joseph (Seffi)~Naor.
\newblock The design of competitive online algorithms via a primal: Dual
  approach.
\newblock {\em Found. Trends Theor. Comput. Sci.}, 3(2\&\#8211;3):93--263,
  February 2009.

\bibitem{LTR}
Zhe Cao, Tao Qin, Tie yan Liu, and Hang Li.
\newblock Learning to rank: from pairwise approach to listwise approach.
\newblock In {\em In Proc. ICML’07}, pages 129--136, 2007.

\bibitem{Chapelle}
Olivier Chapelle and Ya~Zhang.
\newblock A dynamic bayesian network click model for web search ranking.
\newblock In {\em Proceedings of the 18th International Conference on World
  Wide Web}, WWW '09, pages 1--10, New York, NY, USA, 2009. ACM.

\bibitem{LTR_online}
Sougata Chaudhuri and Ambuj Tewari.
\newblock Online learning to rank with top-k feedback.
\newblock {\em CoRR}, abs/1608.06408, 2016.

\bibitem{Chen}
Xiao~Alison Chen and Zizhuo Wang.
\newblock A dynamic learning algorithm for online matching problems with
  concave returns.
\newblock {\em European Journal of Operational Research}, 247(2):379 -- 388,
  2015.

\bibitem{Devanur}
Nikhil~R. Devanur and Thomas~P. Hayes.
\newblock The adwords problem: Online keyword matching with budgeted bidders
  under random permutations.
\newblock In {\em In Proc. 10th Annual ACM Conference on Electronic Commerge
  (EC}, 2009.

\bibitem{Devanur2}
Nikhil~R. Devanur, Zhiyi Huang, Nitish Korula, Vahab~S. Mirrokni, and Qiqi Yan.
\newblock Whole-page optimization and submodular welfare maximization with
  online bidders.
\newblock In {\em Proceedings of the Fourteenth ACM Conference on Electronic
  Commerce}, EC '13, pages 305--322, New York, NY, USA, 2013. ACM.

\bibitem{Mehta}
Aranyak Mehta, Amin Saberi, Umesh Vazirani, and Vijay Vazirani.
\newblock Adwords and generalized on-line matching.
\newblock In {\em In FOCS ’05: Proceedings of the 46th Annual IEEE Symposium
  on Foundations of Computer Science}, pages 264--273. IEEE Computer Society,
  2005.

\bibitem{comp_geom}
Peter Orlik and Hiroaki Terao.
\newblock {\em Arrangements of Hyperplanes}.
\newblock Springer-Verlag, 1992.

\bibitem{Hungarian}
A.~Schrijver.
\newblock {\em Combinatorial Optimization - Polyhedra and Efficiency}.
\newblock 2003.

\end{thebibliography}

\end{document}